\documentclass[12pt]{article}

\usepackage{amsfonts}
\usepackage{amssymb}
\usepackage{amsthm}
\usepackage{amsmath}
\usepackage{mathrsfs}
\usepackage{stmaryrd}
\usepackage[latin1]{inputenc}
\usepackage[all]{xy}
\usepackage[left=1.20in,right=1.20in]{geometry}

\newcommand{\ind}{\textrm{Ind}}
\newcommand{\res}{\textrm{Res}}

\newcommand{\ifl}{\textrm{Inf}}
\newcommand{\dtr}{\textnormal{det}}

\newcommand{\rac}{\mathbb Q}

\newcommand{\ent}{\mathbb Z}
\newcommand{\las}{\mathcal{S}}
\newcommand{\bizlie}[1]{\mathop{\,\raisebox{-.5ex}{$\widehat{\raisebox{.9ex}{\rule{2.5ex}{.07ex}}}_{#1}$}\,}}

\newcommand\boitext[1]{\makebox[0pt][l]{$\scriptstyle #1$}}
\def\xar[#1]{\ar@{-}[#1]}

\theoremstyle{plain}
\newtheorem{teo}{Theorem}[section]
\newtheorem{prop}[teo]{Proposition}

\newtheorem{lema}[teo]{Lemma}
\newtheorem{conj}[teo]{Conjecture}

\theoremstyle{definition}
\newtheorem{defi}[teo]{Definition}
\newtheorem{nota}[teo]{Notation}

\theoremstyle{remark}

\newtheorem{obs}[teo]{Observation}

\title{Computing Whitehead groups using genetic bases}% of metacyclic $p$-groups for odd $p$}
\author{Nadia Romero\footnote{\texttt{nadia.romero@ugto.mx}}\\ 
\begin{small}
Departamento de Matem\'aticas,
\end{small}\\
\begin{small}
Universidad de Guanajuato, Mexico.
\end{small}
}
\date{ }
\begin{document}
\maketitle

\begin{abstract}
 We combine results about Whitehead groups of finite groups with results about genetic bases of finite $p$-groups to compute the Whitehead groups of some metacyclic $p$-groups. Let $C_{p^n}$ denote a cyclic group of order $p^n$ for $p$ an odd prime and $n$ a positive integer. We present a conjecture for the torsion part of the Whitehead group of $C_{p^n}\times C_{p^n}$, and we show that the torsion part of the Whitehead group of the semidirect product of $C_{p^{n-1}}$ by $C_p$ is isomorphic to $C_p^{(n-2)(p-1)}$. The techniques we use can be applied to any abelian $p$-group and to many other $p$-groups for $p$ odd.
\end{abstract}

\section*{Introduction}

The state-of-the-art techniques to calculate Whitehead groups of finite groups are gathered in a book by Robert Oliver of 1988 \cite{bob}. Since then, new results that describe the Whitehead group of a given group usually appear merged in algebraic topology articles that have a different purpose  (see for example \cite{john}, \cite{magu} and \cite{ushi}), no systematic approach to achieve new results  has been done in recent years. Nonetheless, as can be seen in the cited articles, the need to calculate the Whitehead group of certain finite groups remains, and is, in general, not an easy problem. In this article we apply results about \textit{genetic bases} of finite $p$-groups, recently introduced by Serge Bouc, to a specific result from the book of Oliver to do computations that lead us to the following conjecture and theorem.

\theoremstyle{plain}
\newtheorem{conjec}{Conjecture}

\begin{conjec}
Let $p$ be an odd prime and $i$ be a positive  integer. For $n\geq 2i$ define: 
\begin{displaymath}
T_i(n)=(p-1)(2^{E(i)}p^{n-(\left \lfloor{i/2}\right \rfloor +2)}+(n-2i)p^{i-1}), 
\end{displaymath}
%and if $i$ is even
%\begin{displaymath}
%T_i(n)=(p-1)(2p^{n-(i+4)/2}+(n-2i)p^{i-1}). 
%\end{displaymath}
where $E(i)$ is $1$ if $i$ is even and $0$ if it is odd, and $\left \lfloor{i/2}\right \rfloor$ is the floor of $i/2$. %Also define $T_0(n)=0$ for every $n\geq 0$.

For $n\geq 2$, let $W$ be the torsion part of the Whitehead group of $C_{p^n}\times C_{p^n}$. Then, the multiplicity of $C_{p^i}$, for $0<i<n$, in the decomposition of $W$ as a product of cyclic groups, is $T_i(n)$ if $2i\leq n$ and $T_{n-i}(2(n-i))$ otherwise.
\end{conjec}

This conjecture includes the cases $n=2$ and $3$, proved by Oliver in \cite{bob}. We give a sketch of the proof for the case $n=4$. The formulation of this conjecture relies on observations made for some examples through the use of a GAP \cite{gap} program.  

It is worth noticing that the calculations made by the GAP program as well as the techniques we use here to prove the cases $n=2$, $3$ and $4$ can be applied in the exactly same way to any abelian $p$-group for $p$ odd. In particular, this points out that further calculations with GAP should lead to a conjecture for the torsion part of the Whitehead group of any abelian $p$-group with $p$ odd. However, the computer memory needed to perform these calculations increases rapidly. 

We also prove:

\theoremstyle{plain}
\newtheorem{teore}[conjec]{Theorem}

\begin{teore}
Let $M_n(p)$ be the group given by the presentation
\begin{displaymath}
\langle a,\, b\mid a^{p^{n-1}}=1=b^p,\ b^{-1}ab=a^{p^{n-2}+1}\rangle
\end{displaymath}
for $n\geq 3$ and $p$ odd. Then the torsion part of the Whitehead group of $M_n(p)$ is isomorphic to $C_p^{(n-2)(p-1)}$.
\end{teore}

The case $n=3$ of this theorem was already treated by Oliver in Example 9.9 of \cite{bob}.

The techniques used to prove this result can be generalized to any $p$-group, with $p$ odd, having a normal abelian subgroup with cyclic quotient. In fact, if $P$ is a $p$-group satisfying these conditions, we observe that if  $\{S_1,\ldots ,S_k\}$ is a genetic basis for $P$ with $S_1=P$, then the torsion part of the Whithead group of $P$ is isomorphic to a quotient of $\prod_{i=2}^k N_P(S_i)/S_i$, where $N_P(S_i)$ is the normalizer of $S_i$ in $P$. The definition of genetic basis will be explained in the next section. 

\section{Preliminaries}
 
All groups are finite and written multiplicatively.

The results of this section are taken from Oliver \cite{bob} and Bouc \cite{bouc}. 
 
\subsection{The Whitehead group of a finite group}

Let $R$ be an associative ring with unit. The \textit{infinite general linear group of $R$}, denoted by $GL(R)$, is the direct limit of the inclusions $GL_n(R)\rightarrow GL_{n+1}(R)$ as the upper left block matrix. Now, for each $n>0$, denote by $E_n(R)$ the subgroup of $GL_n(R)$ generated by all elementary $n\times n$-matrices. By taking the direct limit as before but for the groups $E_n(R)$, we obtain a subgroup of $GL(R)$, denoted by $E(R)$. Whitehead's Lemma states that $E(R)$ is equal to the derived subgroup of $GL(R)$. The group $K_1(R)$ is defined as the abelianization of $GL(R)$, which is then equal to $GL(R)/E(R)$.

If $G$ is a finite group and we take $R$ as the group ring $\ent G$, then elements of the form $\pm g$ for $g\in G$ can be regarded as invertible $1\times1$-matrices over $\ent G$ and hence they represent elements in $K_1(\ent G)$. Let $H$ be the subgroup of $K_1(\ent G)$ generated by classes of elements of the form $\pm g$ with $g\in G$. The \textbf{Whitehead group} of $G$ is defined as $Wh(G)=K_1(\ent G)/H$. In the words of Oliver: by construction, $Wh(G)$ measures the obstruction to taking an arbitrary invertible matrix over $\ent G$ and reducing it to some $\pm g$ via a series of elementary operations.

The groups $K_1(\ent G)$ and $Wh(G)$ are finitely generated abelian groups. Their rank was described by H. Bass as in the next theorem, and a first description of the torsion part of $Wh(G)$  is given after the theorem.

\begin{teo}[Theorem 2.6 in \cite{bob}]
\label{rank}
Let $G$ be a group and set $r$ and $q$ as the number of non-isomorphic real and rational, respectively, irreducible representations of $G$.
Then $rk(Wh(G))=rk(K_1(\ent G))=r-q$.
\end{teo}

The group $\mathbf{SK_1}(\ent G)$ is defined as the kernel of the morphism $K_1(\ent G)\rightarrow K_1(\rac G)$. By Theorem 2.5 in \cite{bob} it is a finite group and by a theorem of C. Wall (Theorem 7.4 in \cite{bob}) it is isomorphic to the torsion subgroup of $Wh(G)$. To study $SK_1(\ent G)$, Oliver introduces another group, in the following way: for each prime $p$, let $\hat{\ent}_p$ and $\hat{\rac}_p$ denote the $p$-adic completions of $\ent$ and $\rac$ respectively, and set $SK_1(\hat{\ent}_p G)$ as the kernel of the morphism $K_1(\hat{\ent}_p G)\rightarrow K_1(\hat{\rac}_p G)$. Then define $\mathbf{Cl_1}(\ent G)$ as the kernel of the localization morphism
\begin{displaymath}
l:SK_1(\ent G)\twoheadrightarrow \bigoplus_p SK_1(\hat{\ent}_p G).
\end{displaymath}
The sum on the right is actually finite since $SK_1(\hat{\ent}_p G)=1$ if $p\nmid |G|$, and by Theorem 3.9 in \cite{bob}, the morphism $l$ is onto.

\begin{prop}%[Cor. 7.2 in \cite{bob}]\label{skescl}
%Let $p$ be a prime and $G$ be a $p$-group.  If $G$ has a normal abelian subgroup with cyclic quotient, then $SK_1(\hat{\ent}_p G)=1$
[Proposition 12.7 in \cite{bob})] Let $p$ be a prime and $G$ be any group. If a Sylow $p$-subgroup of $G$ has a normal abelian subgroup with cyclic quotient, then $SK_1(\hat{\ent}_p G)=1$.
\end{prop}

This implies that if $G$ is an abelian group or a metacyclic group, describing the torsion part of $Wh(G)$ amounts to describing the group $Cl_1(\ent G)$.

We now focus on the results we will use when working on $p$-groups with $p$ odd. 

The first ingredient is Roquette's Theorem. We will consider the following more recent version appearing as Theorem 9.4.1 in Bouc \cite{bouc}.

\begin{teo}\label{broq}
Let $p$ be a prime, $P$ a $p$-group and $V$ a simple $\rac P$-module. Then there exists a section $S\trianglelefteqslant T\leqslant P$ and a faithful irreducible $\rac (T/S)$-module $W$ such that:
\begin{enumerate}
\item The module $V$ is isomorphic to $\textnormal{Ind}_{\,T}^{\,P}\textnormal{Inf}_{\,T/S}^{\,T}W$.
\item This isomorphism induces an isomorphism of $\rac$-algebras 
\[\textnormal{End}_{\rac P}V\cong \textnormal{End}_{\rac (T/S)}W.\]
\item The group $T/S$ is cyclic if $p$ is odd and dihedral, semi-dihedral, generalized quaternion or cyclic if $p=2$.
\end{enumerate}
\end{teo}

Each one of the groups appearing in 3 has only one faithful irreducible rational representation. If $p$ is odd  and the quotient $T/S$ is cyclic of order $p^r$, then this representation is   $\rac (\zeta_r)$, where $\zeta_r$ is a primitive $p^{r}$-th root of unity and where a generator of $T/S$ acts by multiplication by $\zeta_r$. In this case, the $\rac$-algebra $\rac (\zeta_r)$ is also isomorphic to $\textrm{End}_{\rac (T/S)}W$.

The following definition is taken from \cite{bob} but not written in its entire generality, we write it specifically for the case $p$ odd.

%For the rest of the section $p$ denotes an odd prime and $P$ is a $p$-group.

\begin{defi}[Definition 9.2 in \cite{bob}]\label{lapsi}
Let $p$ be an odd prime and write $\rac P=\prod_{i=1}^kA_i$, where $A_i$ is simple with irreducible module $V_i$ and center $K_i$. By the previous theorem, for each $1\leq i\leq k$, the field $K_i$ is isomorphic to $\rac (\zeta_{r_i})$, where $\zeta_{r_i}$ is a primitive $p^{r_i}$-th root of unity for some non-negative $r_i$, and $A_i$ is isomorphic to a matrix algebra over $\rac(\zeta_{r_i})$.

Suppose that $V_1$ is the trivial module and consider the abelian group $T=\prod_{i=2}^k\langle\zeta_{r_i}\rangle$. For each $h\in P$, define
\begin{displaymath}
\psi_h:C_P(h)\rightarrow T,\quad \psi_h(g)=(\dtr_{K_i}(g, V_i^h))_{i>1}
\end{displaymath}
where $V_i^h=\{x\in V_i\mid hx=x\}$. Here $V_i^h$ is viewed as a $K_iC_P(h)$-module, so $\dtr_{K_i}(g, V_i^h)$ is the determinant (in $K_i$) of the action of $g$ in $V_i^h$. Since $P$ is a $p$-group, this determinant is in $\langle \zeta_{r_i}\rangle$.
\end{defi}
It is easy to see that $\psi_h$ is a group homomorphism.

The last result of this section tells us how to find the group $Cl_1(\ent P)$.

\begin{teo}[Theorem 9.5 in \cite{bob}]\label{elbueno}
Let $p$ be an odd prime and consider $T$ as in the previous definition. Also, for every $h$ in $P$ define $\psi_h:C_P(h)\rightarrow T$ as before. Then
\begin{displaymath}
Cl_1(\ent P)\cong T/\langle \textnormal{Im}\psi_h\mid h\in P\rangle .
\end{displaymath}
\end{teo}

\begin{obs}\label{lapsiab}
When $p$ is odd and $P$ is abelian, we can write $\rac P=\prod_{i=1}^kK_i$, where the $K_i$ are as before, and consider the characters $\chi_i: P\rightarrow \langle \zeta_{r_i}\rangle$. Then for each $h\in P$, the function $\psi_h:P\rightarrow T$ takes the form
\begin{displaymath}
\psi_h(g)=(\bar{\psi}_{i,h}(g))_{i>1} \textrm{ where }\, \bar{\psi}_{i,h}(g)=\left\{
\begin{array}{lc}
\chi_i(g) & \textrm{if } \chi_i(h)=1\\
1 & \textrm{if } \chi_i(h)\neq 1
\end{array} \right. .
\end{displaymath}
Actually, in this case, we can consider a function $\psi:P\times P\rightarrow T$ that sends a pair $(h,g)$ to $\psi_h(g)$. %Notice that $\psi$ is multiplicative in the second variable.
\end{obs}

\subsection{Genetic bases of $p$-groups}

After the reformulation of Roquette's Theorem, Bouc shows (Lemma 9.4.3 in \cite{bouc}) that if  $S\trianglelefteqslant T\leqslant P$ is one of the sections appearing in the theorem, then one has $T=N_P(S)$. With this, we have the following definition.

\begin{defi}
A subgroup $S$ of a  $p$-group $P$ is called \textbf{genetic} if the section $S\trianglelefteqslant N_P(S)\leqslant P$ satisfies points 2 and 3 of Theorem \ref{broq}, where $W$ is the unique faithful irreducible  $\rac$-module of $N_P(S)/S$  and  $V=\textnormal{Ind}_{\,N_P(S)}^{\,P}\textnormal{Inf}_{\,N_P(S)/S}^{\,N_P(S)}W$. \end{defi}

% A subgroup $S$ of $P$ is then called genetic if $S\trianglelefteqslant N_P(S)\leqslant P$ is one of the sections appearing in Roquette's Theorem, in our case, this gives the next definition. 

%Roquette's Theorem and the results we will see in this section can actually be stated for any prime, having different conditions on the group $N_P(S)/S$ for the prime $2$. For simplicity, in this section we reformulate these results for the case $p$ odd.

%From now on, $p$ denotes an odd prime.

\begin{nota}
Let $P$ be a $p$-group an $S$ be a genetic subgroup of $P$. If $\Phi_{N_P(S)/S}$ denotes the unique irreducible faithful rational representation of $N_P(S)/S$, then we write
\begin{displaymath}
V(S)=\ind_{N_P(S)}^P\ifl_{N_P(S)/S}^{N_P(S)}\Phi_{N_P(S)/S}.
\end{displaymath}
\end{nota}

The following theorem characterizes the genetic subgroups of a $p$-group.

\begin{teo}[Theorem 9.5.6 in \cite{bouc}]\label{cargen}
Let $P$ be a $p$-group and $S$ a subgroup of $P$ such that $N_P(S)/S$ is cyclic if $p$ is odd, or also dihedral, semi-dihedral or generalized quaternion if $p=2$. Then the following conditions are equivalent:
\begin{enumerate}
\item The subgroup $S$ is a genetic subgroup of $P$.
\item If $x\in P$ is such that $^xS\cap N_P(S)\leqslant S$, then $^xS=S$.
\item If $x\in P$ is such that $^xS\cap N_P(S)\leqslant S$ and $S^x\cap N_P(S)\leqslant S$, then $^xS=S$.
\end{enumerate}
\end{teo}

The next result gives us a necessary and sufficient condition for two genetic subgroups $S$ and $T$ of $P$ to produce isomorphic representations, $V(S)\cong V(T)$, it is part of Theorem 9.6.1 in \cite{bouc}.

\begin{teo}\label{cariso}
Let $P$ be a $p$-group and $S$ and $T$ be genetic subgroups of $P$. The following conditions are equivalent:
\begin{enumerate}
\item The $\rac P$-modules $V(S)$ and $V(T)$ are isomorphic.
\item There exists $x\in P$ such that $^xT\cap N_P(S)\leqslant S$ and $S^x\cap N_P(T)\leqslant T$.
\end{enumerate}
If these conditions hold, then in particular the groups $N_P(S)/S$ and $N_P(T)/T$ are isomorphic.
\end{teo}

The relation between groups appearing in point 2 is denoted by $S\bizlie{P}T$. The theorem shows that this relation is an equivalence relation on the set of genetic subgroups of $P$, and we have the following definition.

\begin{defi}[Definition 9.6.11 in \cite{bouc}]
A \textbf{genetic basis} of $P$ is a set of representatives of the equivalence classes of $\bizlie{P}$ in the set of genetic subgroups of $P$.
\end{defi}

\begin{obs}
Suppose that $p$ is odd and that $\{S_1,\ldots ,S_k\}$ is a genetic basis for $P$ with $S_1=P$. Then Theorem \ref{elbueno} says that $Cl_1(P)$ is isomorphic to a quotient of $\prod_{i=2}^k N_P(S_i)/S_i$. %This is a consequence of Theorem \ref{broq}, since 
This is because we can see the module $\Phi_{N_P(S_i)/S_i}\cong \rac (\zeta_{r_i})$, where $r_i=|N_P(S_i)/S_i|$ as actually being generated by a generator of $N_P(S_i)/S_i$ and thus the action of $N_P(S_i)/S_i$ on it can be seen as multiplication on the group.
%the integer $r_i$ in the field $\rac (\zeta_{r_i})$ of Definition \ref{lapsi} is precisely the order of $N_P(S_i)/S_i$.
\end{obs}

\begin{obs}
If $P$ is an abelian group, a genetic subgroup of $P$ is a subgroup with cyclic quotient, and the relation $\bizlie{P}$ becomes equality. So, in this case there is a unique genetic basis $\mathcal{S}$, consisting of the subgroups of $P$ with cyclic quotient. It is not hard to see that the genetic subgroups of $P$ in this case can be obtained as the kernels of the irreducible complex representations of $P$.

In this case, if $p$ is odd and $\{S_1,\ldots S_k\}$ is the genetic basis for $P$, with $S_1=P$, then with the help of Observation \ref{lapsiab}, Theorem \ref{elbueno} can be reformulated in the following way: Take $T=\prod_{i=2}^kP/S_i$, and %Define the function $\psi_{ab}:P\times P\rightarrow T'$ as
for each $h\in P$ define the group homomorphism $\psi_h:P\rightarrow T$
\begin{displaymath}
\psi_h(g)=(\bar{\psi}_{h,\,i}(g))_{i>1} \textrm{ where }\, \bar{\psi}_{h,\,i}(g)=\left\{
\begin{array}{lc}
gS_i & \textrm{if } h\in S_i\\
1 & \textrm{if } h\notin S_i
\end{array} \right. .
\end{displaymath}
%Notice that $\psi_{ab}$ is multiplicative in the second variable.
Then Theorem \ref{elbueno} says that $Cl_1(P)$ is isomorphic to $T/\langle\textnormal{Im}\psi_{h}\mid h\in P\rangle$.
\end{obs}

\section{Genetic bases of two metacyclic groups}

Given a finite metacyclic group $G$ with a presentation
\begin{equation*}\label{pres}
G=\langle x,y\mid x^m=1,\, y^s=x^t,\, y^{-1}xy=x^r\rangle
\end{equation*}
where the integers $m$, $s$, $t$ and $r$ satisfy $r^s\equiv 1\ (\textsl{mod}\ m)$ and $m|t(r-1)$, one can find the cardinality of a genetic basis of $G$ in the following way: Basmaji in \cite{bas} describes all the non-isomorphic irreducible complex representations of $G$, let us denote them by Irr$\, G$. Next, if $G$ has exponent $n$ and $\varepsilon$ is a primitive $n$-th root of unity, then $Gal(n)$, the Galois group of the extension $\rac(\varepsilon):\rac$, acts on Irr$\,G$, %by sending a character $\chi$ to $\chi^{\sigma}$. 
%If $\mathcal{T} =\{T_1,\, \ldots T_n\}$ are the $\lak(\varepsilon)$- irreducible representations of $G$, 
and the number of non-isomorphic irreducible $\rac$-representations of $G$ is  the number of orbits of this action (section 70 of Curtis and Reiner \cite{curtis1}). This method can be easily used for any of our two examples, but of course, for the first example there are other methods.  %In our case,  the action of $Gal(n)$ in Irr$\,G$ is given by sending a 

The following descriptions work for any prime $p$, but remember that in the next section we will only be considering $p$ odd.

\subsection{$C_{p^n}\times C_{p^n}$}

Suppose $n$ and $m$ are positive integers with $m\geq n$. The size of the genetic basis of $C_{p^n}\times C_{p^m}$ is equal to the number of subgroups with cyclic quotient, which is also, by duality of abelian groups, equal to the number of cyclic subgroups. This number can be found, for example, using Goursat's Lemma (Theorem 2.3 in Sim \cite{sim}), it is equal to
\begin{displaymath}
p^{n}(m-n+1)+2\frac{p^n-1}{p-1}.
\end{displaymath}
Actually, this number is equal to the sum $1+(p+1)\sum_{i=0}^{n-1}p^i+p^n(m-n)$, which tells us how many groups of each index we will have in the genetic basis. Indeed, in the genetic basis for this group we will have of course 1 group of index 1, $p+1$ groups of index $p$, then $(p+1)p^i$ groups of index $p^{i+1}$ for every $1\leq i\leq n-1$ and then $p^n$ groups of index $p^j$ for every $j$ between $n$ and $m$. Clearly, any group of index bigger than $p^m$ has non-cyclic quotient so these are all the possibilities we have for the indices.  Since we are interested in the case $m=n$, we will give a precise description of the genetic basis  for this case.  This description can also be obtained from Lemme 2 (Chapitre VII) in Serre \cite{arith}.

\newpage

For the rest of this section, we fix $P$ as being the group $C_{p^n}\times C_{p^n}$. Suppose $P$ has generators $a$ and $b$, that is $P=\langle a\rangle \times \langle b\rangle$ and let $\mathcal{S}$ be the genetic basis of $P$. Then  $\las$ consists of the following subgroups:
\begin{displaymath}
\begin{array}{cccc}
 & P &  & \\
\langle a, b^p\rangle & \langle a^{i}b, a^p\rangle & \langle b, a^p\rangle & 1\leq i\leq p-1\\
\langle ab^{jp}, b^{p^2}\rangle & \langle a^{jp+i}b, a^{p^2}\rangle & \langle a^{jp}b, a^{p^2}\rangle & 0\leq j\leq p-1\\
 & \vdots & & \vdots\\
\langle ab^{kp^{n-1}+\ldots +jp}\rangle & \langle a^{kp^{n-1}+\ldots + i}b\rangle & \langle a^{kp^{n-1}+\dots +jp}b\rangle & 0\leq k\leq p-1. 
\end{array}
\end{displaymath}
These are all distinct subgroups, they have all cyclic quotient in $P$ and we have $1+(p+1)\sum_{i=0}^{n-1}p^i$ of them, so this is the genetic basis of $P$. That is, a group in $\las$ of index $p^s$ for $0\leq s\leq n$, is given by $\langle a^{\nu}b, a^{p^s}\rangle$, where $\nu$ is in the set of smallest non-negative representatives of $\ent/p^s\ent$, or by $\langle ab^{\mu}, b^{p^s}\rangle$ where $1\leq s\leq n$ and $\mu$ is in the same set as $\nu$ but it is a multiple of $p$. We write $\mu$ and $\nu$ in base $p$ expansion since this helps us to keep track of the inclusions between subgroups. Notice that the only cyclic groups of the basis are those of index $p^n$.

 The following diagrams illustrate how the genetic bases look like for the example $p=3$ and $n=1,\,2$ and $3$. Let $\las_n$ be the genetic basis of $P$ with these values. The column on the left indicates the index of the groups and how many of them are in $\las_n$. Black dots are the groups in $\las_n$ and white circles indicate the rest of groups of the given index. Notice that the diagrams are symmetric with respect to the line of groups of order (and index) $p^n$.

$n=1:$
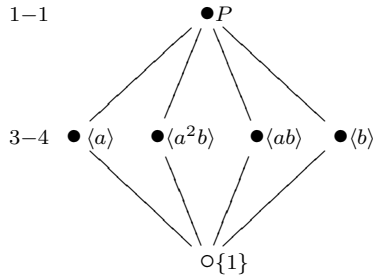
\begin{displaymath}
\xymatrix@C=0.4pt@R=33pt{
\scriptstyle{1-1}& & & & & & & \bullet\boitext{P} & & & & & & \\
\scriptstyle{3-4}& \bullet\boitext{\,\langle a\rangle}\xar[urrrrrr]\xar[drrrrrr]& & & &  \bullet\boitext{\langle a^2b\rangle}\xar[urr]\xar[drr]&  & & &\bullet\boitext{\langle ab\rangle}\xar[ull]\xar[dll]& &  & &\bullet\boitext{\langle b\rangle}\xar[ullllll]\xar[dllllll]\\
& & & & & & & \circ\boitext{\{1\}}& & & & & &  
}
\end{displaymath} 
 
$n=2:$
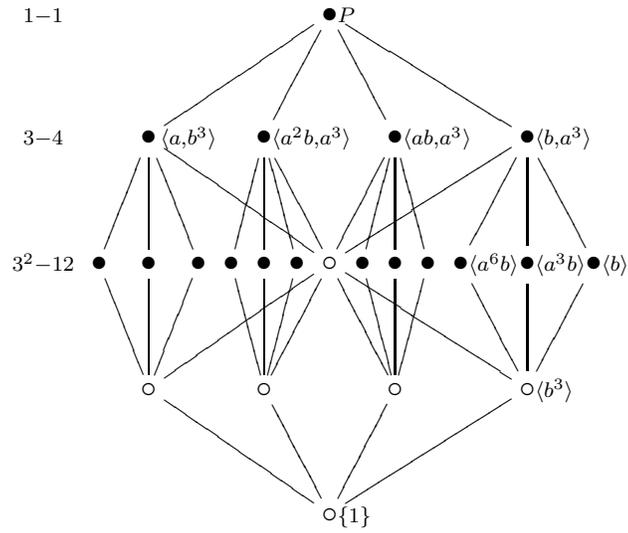
\begin{displaymath}
\xymatrix@C=0.4pt@R=33pt{
\scriptstyle{1-1}& & & && & & & & \bullet\boitext{P}& & & & & & & & \\
\scriptstyle{3-4}& & &\bullet\boitext{\,\langle a,b^3\rangle}\xar[urrrrrr]& & & & \bullet\boitext{\langle a^2b, a^3\rangle}\xar[urr]& & & &\bullet\boitext{\langle ab, a^3\rangle}\xar[ull]& & & & &\bullet\boitext{\langle b,a^3\rangle}\xar[ulllllll] & & & \\
\scriptstyle{3^2-12}&\bullet\xar[urr]\xar[drr] & & \bullet\xar[u]\xar[d] & &\bullet\xar[ull]\xar[dll] &   \bullet\xar[ur]\xar[dr] & \bullet\xar[u]\xar[d] &\bullet\xar[ul]\xar[dl] & \xar[ullllll]\xar[dllllll]\xar[ull]\xar[dll]\circ\xar[urr]\xar[drr]\xar[urrrrrrr]\xar[drrrrrrr] & \bullet\xar[ur]\xar[dr]  &\bullet\xar[u]\xar[d] & \bullet\xar[ul]\xar[dl]  & \bullet\boitext{\langle a^6b\rangle}\xar[urrr]\xar[drrr] & & &\bullet\boitext{\langle a^3b\rangle}\xar[u]\xar[d]& & &  \bullet\boitext{\langle b\rangle}\xar[ulll]\xar[dlll]\\
& & &\circ\xar[drrrrrr]& & & & \circ\xar[drr]& & & &\circ\xar[dll]& & & & &\circ\boitext{\langle b^3\rangle}\xar[dlllllll] & & & \\
 & & & && & & & & \circ\boitext{\{1\}}& & & & & & & & 
}
\end{displaymath} 
 
$n=3:$ 
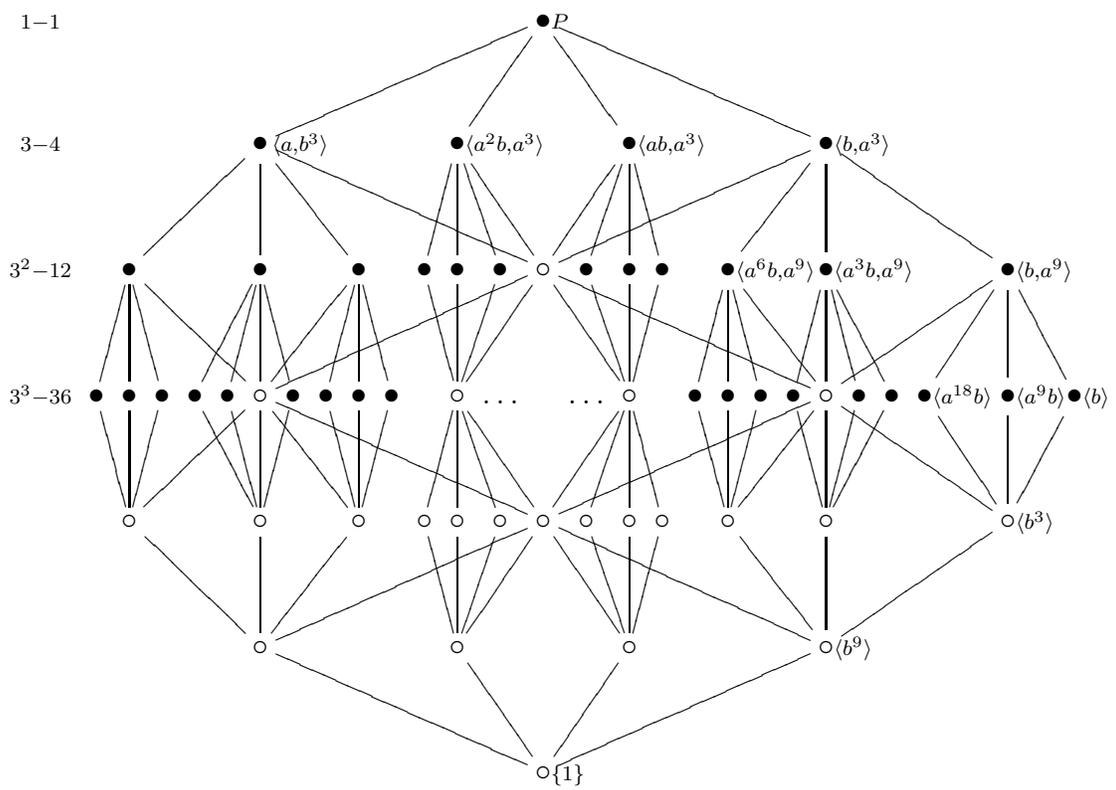
\begin{displaymath}
\xymatrix@C=0.4pt@R=33pt{
\scriptstyle{1-1}& & & & & & & & & & & & & &\bullet\boitext{P}& & & & & & & & & & & & & & &\\
\scriptstyle{3-4}&& & & & & \bullet\boitext{\,\langle a,b^3\rangle}\xar[urrrrrrrr]& & & & & &\bullet\boitext{\langle a^2b, a^3\rangle}\xar[urr]& & & &\bullet\boitext{\langle ab, a^3\rangle}\xar[ull]& & & & & &\bullet\boitext{\langle b,a^3\rangle}\xar[ullllllll] & & & & & & &\\
\scriptstyle{3^2-12}&&\bullet\xar[urrrr] & & & &\bullet\xar[u] & & & \bullet\xar[ulll] & &  \bullet\xar[ur] & \bullet\xar[u] &\bullet\xar[ul] & \xar[ullllllll]\xar[ull]\circ\xar[urr]\xar[urrrrrrrr] & \bullet\xar[ur]  &\bullet\xar[u] & \bullet\xar[ul]  & &\bullet\boitext{\langle a^6b, a^9\rangle}\xar[urrr] & & &\bullet\boitext{\langle a^3b,a^9\rangle}\xar[u]&&&& & & & \bullet\boitext{\langle b,a^9\rangle}\xar[ulllllll]& & \\
\scriptstyle{3^3-36}&\bullet\xar[ur]\xar[dr] &  \bullet\xar[u]\xar[d] &\bullet\xar[ul]\xar[dl]  & \bullet\xar[urr]\xar[drr] &  \bullet\xar[ur]\xar[dr] &\xar[urrr]\xar[u]\xar[drrr]\xar[d]\circ\xar[ullll]\xar[urrrrrrrr]\xar[dllll]\xar[drrrrrrrr] &\xar[ul]\xar[dl]\bullet & \bullet\xar[ur]\xar[dr] &  \bullet\xar[u]\xar[d] &\bullet\xar[ul]\xar[dl]  & & \xar[ul]\xar[u]\xar[dl]\xar[d]\circ\xar[ur]\xar[urr]\xar[dr]\xar[drr] & \ldots & &\ldots & \xar[ull]\xar[ul]\xar[u]\xar[dll]\xar[dl]\xar[d]\circ\xar[ur]\xar[dr] & & \bullet\xar[ur]\xar[dr] &  \bullet\xar[u]\xar[d] &\bullet\xar[ul]\xar[dl] &\bullet\xar[ur]\xar[dr] &\xar[ullllllll]\xar[ulll]\xar[dllllllll]\xar[dlll]\circ\xar[u]\xar[urrrrrrr]\xar[d]\xar[drrrrrrr] & \bullet\xar[ul]\xar[dl] &\bullet\xar[ull]\xar[dll] &\bullet\xar[urrrr]\xar[drrrr] \boitext{\langle a^{18}b\rangle}& & && \bullet\xar[u]\xar[d]\boitext{\langle a^9b\rangle} && &\bullet\boitext{\langle b \rangle}\xar[ulll]\xar[dlll]\\
&&\circ\xar[drrrr] & & & & \circ\xar[d] & & & \circ\xar[dlll] & &  \circ\xar[dr] & \circ\xar[d] &\circ\xar[dl] & \xar[dllllllll]\xar[dll]\circ\xar[drr]\xar[drrrrrrrr] & \circ\xar[dr]  &\circ\xar[d] & \circ\xar[dl]  & &\circ\xar[drrr] & & &\circ\xar[d]& &&& & & & \circ\xar[dlllllll]\boitext{\langle b^3\rangle}& & \\
&& & & & & \circ\xar[drrrrrrrr]& & & & & &\circ\xar[drr]& & & &\circ\xar[dll]& & & & & &\circ\xar[dllllllll]\boitext{\langle b^9\rangle} & & & & & & &\\
& & & & & & & & & & & & & &\circ\boitext{\{1\}}& & & & & & & & & & & & & & &
}
\end{displaymath}
In this case, $P$ has 4 groups of index $3$, which of course are all in $\las_3$. Each one of them has in turn 4 groups of index 3 (and so their index in $P$ is $3^2$), but among these 4 groups one is contained in all the groups of index $3$ and does not have cyclic quotient in $P$, it is $\langle a^3, b^3\rangle$, represented by the only white dot of index $3^2$. The other $3$ groups have cyclic quotient in $P$ and so we obtain $3\cdot 4=12$ groups of index $3^2$ in  $\las_3$.  A similar phenomenon occurs for these groups and we obtain $3\cdot 12=36$ groups of index $3^3$ in $\las_3$.

\subsection*{The group $M_n(p)$}

The group $M_n(p)$ has a presentation
\begin{displaymath}
\langle a,\, b\mid a^{p^{n-1}}=1=b^p,\ b^{-1}ab=a^{p^{n-2}+1}\rangle
\end{displaymath}
for $n\geq 3$.

\begin{lema}
\label{tam}
The cardinality of a genetic basis of $M_n(p)$ is $(n-2)p+3$.
\end{lema}
\begin{proof}
By the results in \cite{bas}, we have that complex representations of $M_n(p)$ are indexed by pairs $(\zeta,\, \theta)$, where $\zeta$ runs through a set of representatives of the natural action of $r=p^{n-2}+1$ on the group of $p^{n-1}$-th roots of unity, and if $\zeta$ is a root of order $p^i$ with $0\leq i\leq n-1$ then $\theta$ runs through the $(p/s_i)$-th roots of unity, where $s_i$ is the order of $r$ modulo $p^i$. Observe that if $i<n-1$, then $r$ fixes $\zeta$ and $s_i=1$. For $i=n-1$, by Lemma 2.1 in \cite{sim}, $s_{n-1}=p$. The exponent of $M_n(p)$ is $p^{n-1}$ and the Galois group of the extension of $\rac$ by a primitive $p^{n-1}$-th root of unity is isomorphic to $(\ent /p^{n-1}\ent)^{\times}$. If $\bar{q}$ is in $(\ent/p^{n-1}\ent)^{\times}$, then it sends $(\zeta,\, \theta)$ to $(\zeta^q,\, \theta^q)$. As said at the beginning of the section, the cardinality of a genetic basis is equal to the number of orbits of this action, which can now be easily computed: if $\zeta=1$ then we obtain two orbits, for any other $\zeta$ of order $p^i$ with $1 \leq  i \leq n-2$ we obtain $p$ orbits, and for $\zeta$ of order $p^{n-1}$ we obtain one orbit.

%As can be seen in ?????? the cardinality of a genetic basis of $M_n(p)$ is equal to the sum $\sum_{k|m}Or(k)$, were $m=p^{n-1}$ and $Or(k)$ is described in Proposition. In the notation of this proposition, $s=p$, the parameter $t$ can be taken as $0$ or $p^{n-1}$ and $r=1+p^{n-2}$. Continuing with the notation there, for each divisor $k$ of $m$, we have $k_t=k$. Now, the order of $r$ modulo $k$, that is $s_k$, is $1$ for $k=1,\ldots p^{n-2}$ and $p$ for $k=p^{n-1}$. With this we can now calculate $Or(k)$ for each $k$. For $k=1$ we clearly have that $Or(1)=2$, for $k=p,\ldots p^{n-2}$ it is easy to see that $Or(k)=p$ and for $k=p^{n-1}$ we clearly have $Or(k)=1$.
\end{proof}

To find a genetic basis of $P$ we will find $(n-2)p+3$ groups $S$ with cyclic quotient $N_p(S)/S$, that satisfy condition 2 of Theorem \ref{cargen} and which are not linked under $\bizlie{P}$. 

A genetic basis $\mathcal{S}$ of $P$ consists of the following groups:
\begin{displaymath}
\begin{array}{ccccccc}
& & P &  & & &\\
\langle a\rangle & \langle a^{i}b\rangle & \langle a^p, b\rangle & & & & 1\leq i\leq p-1\\
& & \langle a^{jp}b\rangle & \langle a^{p^2}, b\rangle & & & 1\leq j\leq p-1\\
%& & & \langle a^{kp^2}b\rangle & \langle a^{p^3}, b\rangle & & 1\leq k\leq p-1\\
& &\ldots & & & &\\
& & & \langle a^{kp^{n-3}}b\rangle &\langle a^{p^{n-2}},\, b\rangle & &1\leq k\leq p-1\\  
& & & & & \langle b\rangle &
\end{array}
\end{displaymath}

Before proving this is a genetic basis, we see a genetic basis for $M_4(3)$ in the following, more schematic way:
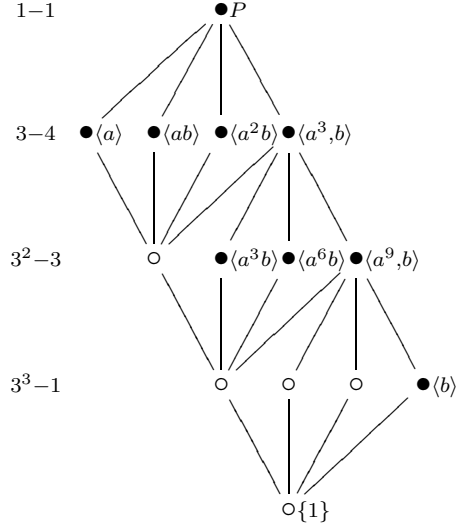
\begin{displaymath}
\xymatrix@C=0.5pt@R=33pt{
\scriptstyle{1-1} & & & & & & & \bullet\boitext{P} & & & & & & & & & \\
\scriptstyle{3-4} & \bullet\boitext{\langle a\rangle}\xar[urrrrrr] & & & \bullet\boitext{\langle ab\rangle}\xar[urrr]& & &\bullet\boitext{\langle a^2b \rangle}\xar[u]& & &\bullet\boitext{\langle a^3, b\rangle}\xar[ulll] & & & & & & \\
\scriptstyle{3^2-3} & & & & \xar[ulll]\xar[u]\circ\xar[urrr]\xar[urrrrrr] & & & \bullet\boitext{\langle a^3b\rangle}\xar[urrr] & & &\bullet\boitext{\langle a^6b\rangle}\xar[u]& & & \bullet\boitext{\langle a^9,b\rangle}\xar[ulll]& & &\\
\scriptstyle{3^3-1}& & & & & & & \xar[ulll]\xar[u]\circ\xar[urrr]\xar[urrrrrr] & & &\circ\xar[urrr] & & & \circ\xar[u] & & & \bullet\boitext{\langle b\rangle}\xar[ulll]\\
& & & & & & & & & & \xar[ulll]\xar[u]\circ\boitext{\{1\}}\xar[urrr]\xar[urrrrrr] & & & & & &
}
\end{displaymath}

We will write $B$ for $\langle b\rangle$. For the groups $\langle a^{p^i}\rangle$ we will write $A_i$ and $Q_{ij}$ for $\langle a^{jp^i}b\rangle$ with $1\leq j\leq p-1$ and $0\leq i\leq n-2$.

\begin{prop}
The groups $B$, $A_0$, $A_iB$, $Q_{ij}$ with $0\leq i\leq n-3$ and $1\leq j\leq p-1$,  and $A_{n-2}B$ form a genetic basis for $M_n(p)$.
\end{prop}
\begin{proof}
For any positive integer $u$, we have $a^uba^{-u}=ba^{up^{n-2}}$. Since $a^{p^{n-2}}$ is a central element, this implies that for any $i$ and $j$ the conjugate $a^ua^{jp^i}ba^{-u}$ is equal to $a^{up^{n-2}}a^{jp^i}b$. Since all the groups $A_{n-2}B$, $A_iB$ and $Q_{ij}$ contain $a^{p^{n-2}}$, we have  that they are all normal in $P$, and clearly they have cyclic quotient. Since $A_0$ is normal too, by Theorem \ref{cargen}, they are all genetic groups. Now, $B$ is normal in $A_1B$ and not in $P$, so its normalizer is $A_1B$. The first line of this paragraph also shows that the conjugates of $B$ are the groups $\langle a^{vp^{n-2}}b\rangle$ with $1\leq v\leq p-1$, which are all contained in $A_1B$, so by Theorem \ref{cargen}, $B$ is genetic.

According to Theorem \ref{cariso}, we need to verify that given two genetic subgroups $S$ and $T$ such that $N_P(S)/S$ and $N_P(T)/T$ have the same order, then $S$ and $T$ are not related under $\bizlie{P}$. Now, if $N_P(S)$ and $N_P(T)$ are both equal to $P$, then $S$ and $T$ are related under $\bizlie{P}$ if and only if they are equal. On the other hand, it is not hard to see that the groups that are related to $B$ are its conjugates. This finishes the proof. 
\end{proof}

\section{Finding the Whitehead groups of $C_{p^n}\times C_{p^n}$ and $M_n(p)$}

In this section we suppose that $p$ is an odd prime.

Recall, for example from Exercise 13.9 on Serre \cite{serre}, that if $G$ is a group of odd order and $c$ is the number of irreducible non-isomorphic complex representations of $G$, then $(c+1)/2$ is the number of irreducible non-isomorphic real representations of $G$. With this and Theorem \ref{rank}, we obtain the following two results about the rank of the Whitehead group for the groups in question.

\begin{lema}
The rank of the Whitehead group of $C_{p^n}\times C_{p^n}$ is 
\begin{displaymath}
\frac{kp^{2n}-(p+1)p^n+k+2}{p-1},
\end{displaymath}
where $p-1=2k$.
\end{lema}
\begin{proof}
The number of irreducible non-isomorphic real representations of $C_{p^n}\times C_{p^n}$ is $(p^{2n}+1)/2$ and, as we saw in the previous section, the number of irreducible non-isomorphic rational representations is $p^n+2\frac{p^n-1}{p-1}$ which is equal to $\frac{p^n(p+1)-2}{p-1}$.
\end{proof}

%Since the free part of the Whitehead group of a finite group depends on the number of irreducible real and complex representations of the group, we first find these numbers in terms of the parameters of the presentation given to the group. To do this we use the results of  Appendix I.

\begin{lema}
The rank of the Whitehead group of $M_n(p)$ is 
\begin{displaymath}
\frac{(p-1)p^{n-3}+p^{n-1}-2(n-2)p-5}{2}.
\end{displaymath}
\end{lema}
\begin{proof}
From the proof of Lemma \ref{tam}, we have that the number of complex non-isomorphic irreducible representations of $M_n(p)$ is
\begin{displaymath}
\sum_{i=0}^{n-2}\Phi(p^i)p+\frac{\Phi(p^{n-1})}{p},
\end{displaymath}
where $\Phi$ denotes the Euler function. This is equal to $p^{n-3}(p-1)+p^{n-1}$. On the other hand, from Lemma \ref{tam}, the number of rational non-isomorphic irreducible representations is $(n-2)p+3$.
\end{proof}

We proceed now to calculate the torsion parts. The use of the diagrams on the previous section will be of great help.

\subsection{$SK_1$ of $C_{p^n}\times C_{p^n}$}

If $\mathcal{S}=\{S_1,\ldots ,S_k\}$ is the genetic basis of $P=C_{p^n}\times C_{p^n}$, then according to the first section, we have to calculate the quotient of $T=\prod_{i=1}^k P/S_i$ over the group generated by the images of $\psi_h$, where $\psi_h:P\rightarrow T$
\begin{displaymath}
\psi_h(g)=(\bar{\psi}_{h,\,i}(g))_{i>1} \textrm{ with }\, \bar{\psi}_{h,\,i}(g)=\left\{
\begin{array}{lc}
gS_i & \textrm{if } h\in S_i\\
1 & \textrm{if } h\notin S_i
\end{array} \right. .
\end{displaymath}
Since each $\psi_{h}$ is a group homomorphism, we can take $g$ running on a set of generators of $P$.

The cases $n=1$, $2$ and $3$ were already given by Oliver in \cite{bob}, but we made them here again to show how the procedure works in general.

\subsubsection*{$C_p\times C_p$}
Let $P=C_p\times C_p$ and suppose it has generators $a$ and $b$, then the genetic basis is $\mathcal{S}=\{P,\, \langle a\rangle,\, \langle a^ib\rangle,\, \langle b\rangle\}$ with $1\leq i\leq p-1$.

The group $T$ is $P/\langle a\rangle\times \prod_{i=1}^{p-1}P/\langle a^ib\rangle\times P/\langle b\rangle$, and isomorphic to $C_p^{p+1}$. To calculate the corresponding quotient, we make the following table 
\begin{displaymath}
\begin{array}{r|ccc}
 & P/\langle a\rangle & P/\langle a^ib\rangle & P/\langle b\rangle\\
 \hline
\psi_b(a) & 1 & 1 & a\langle b\rangle\\
\psi_b(b) & 1 & 1 & 1\\
\psi_{a^ib}(a) & 1 & a\langle a^ib\rangle & 1\\
\psi_{a^ib}(b) & 1 & b\langle a^ib\rangle & 1\\
\psi_a(a) & 1 & 1 & 1\\
\psi_a(b) & b\langle a\rangle & 1 & 1\\
\psi_1(a) & 1 & a\langle a^ib\rangle & a\langle b\rangle \\
\psi_1(b) & b\langle a\rangle & b\langle a^ib\rangle & 1
\end{array}
\end{displaymath}
Observe that the first two rows give the same result for any non-trivial element in $\langle b\rangle$. The same happens for the rows corresponding to $\langle a^ib\rangle$ and for those corresponding to $\langle a\rangle$. So, from this table, actually from rows 1, 3 and 6, we see that $\langle \textnormal{Im}\psi_h\mid h\in P\rangle$ is equal to $T$ and so $Cl_1(\ent P)=SK_1(\ent P)$ is trivial.

In the following examples the diagrams of the previous section will be a useful guide to understand the tables.

\subsubsection*{$C_{p^2}\times C_{p^2}$}

Let $P=C_{p^2}\times C_{p^2}$, and suppose it has generators $a$ and $b$. The genetic basis of $P$ is $\{P, \,\langle a,\, b^p\rangle,\, \langle a^ib,\, a^p\rangle,\, \langle ab^{jp}\rangle,\, \langle a^{jp+i}b\rangle\}$ with $0\leq i,\,j\leq p-1$. The group $T$ in this case is then isomorphic to $C_p^{p+1}\times C_{p^2}^{p(p+1)}$.

We make again a table. In the isolated column on the left we write, for each $\psi_h$, the order of the element $h$. Columns 1 and 2 correspond to quotient groups of order $p$, while  columns 3 and 4 correspond to quotient groups of order $p^2$. We will write only a bar over the elements to indicate their class in the quotient.

\begin{displaymath}
\begin{array}{lr|cccc}
 && \frac{P}{\langle a,\, b^p\rangle} & \frac{P}{\langle a^ib,\, a^p\rangle} & \frac{P}{\langle ab^{jp}\rangle} & \frac{P}{\langle a^{jp+i}b\rangle} \\
 \cline{2-6}
p^2&\psi_{a^{jp+i}b}(a) &1 &\bar{a} &1 & \bar{a}\\
&\psi_{a^{jp+i}b}(b) &1 &\bar{b} &1 & \bar{b}\\
&\psi_{ab^{jp}}(a) &1 &1 &\bar{a} & 1\\
&\psi_{ab^{jp}}(b) &\bar{b} & 1&\bar{b} & 1\\
p&\psi_{a^{ip}b^p}(a) & 1 &\bar{a} (\forall i) &1 & \bar{a} (\forall j)\\
&\psi_{a^{ip}b^p}(b) &\bar{b} &\bar{b} (\forall i) &1 & \bar{b} (\forall j)\\
&\psi_{a^p}(a) &1 &\bar{a}(\forall i) &\bar{a}(\forall j) &1\\
&\psi_{a^p}(b) &\bar{b} &\bar{b}(\forall i) &\bar{b}(\forall j) &1\\
1&\psi_{1}(a) & 1 &\bar{a}(\forall i) &\bar{a}(\forall j) & \bar{a}(\forall j\forall i)\\
&\psi_{1}(b) & \bar{b} &\bar{b}(\forall i) &\bar{b}(\forall j) & \bar{b}(\forall j\forall i)
\end{array}
\end{displaymath}
%\begin{displaymath}
%\begin{array}{cr|cccc}
%%&&p&p^2&\\
% && P/\langle a,\, b^p\rangle & P/\langle a^ib,\, a^p\rangle & P/\langle ab^{jp}\rangle & P/\langle a^{jp+i}b\rangle \\
% \cline{2-6}
%p^2&\psi_{a^{jp+i}b}(a) &1 &a\langle a^ib,\, a^p\rangle &1 & a\langle a^{jp+i}b\rangle\\
%&\psi_{a^{jp+i}b}(b) &1 &b\langle a^ib,\, a^p\rangle &1 & b\langle a^{jp+i}b\rangle\\
%&\psi_{ab^{jp}}(a) &1 &1 &a\langle ab^{jp}\rangle & 1\\
%&\psi_{ab^{jp}}(b) &b\langle a,\, b^{p}\rangle & 1&b\langle ab^{jp}\rangle & 1\\
%p&\psi_{a^{ip}b^p}(a) & 1 &a\langle a^ib,\, a^p\rangle\, (\forall i) &1 & a\langle a^{jp+i}b\rangle\,(\forall j)\\
%&\psi_{a^{ip}b^p}(b) &b\langle a,\, b^p\rangle &b\langle a^ib,\, a^p\rangle\, (\forall i) &1 & b\langle a^{jp+i}b\rangle\, (\forall j)\\
%&\psi_{a^p}(a) &1 &a\langle a^ib,\, a^p\rangle &a\langle ab^{jp}\rangle &1\\
%&\psi_{a^p}(b) &b\langle a,\, b^p\rangle &b\langle a^ib,\, a^p\rangle &b\langle ab^{jp}\rangle &1\\
%1&\psi_{1}(a) & 1 &a\langle a^ib,\, a^p\rangle &a\langle ab^{jp}\rangle & a\langle a^{jp+i}b\rangle\\
%&\psi_{1}(b) & b\langle a,\, b^p\rangle &b\langle a^ib,\, a^p\rangle &b\langle ab^{jp}\rangle & b\langle a^{jp+i}b\rangle
%\end{array}
%\end{displaymath}
Notice that the class of $b$ may be 1 in some cases. In the first four rows there is at most one entry of order $p$, and one entry of order $p^2$, they correspond to $i$ and to $jp+i$ or to $j$, depending on the element $h$ in $\psi_h$. In the next four rows, we have blocks of size $p$ of classes of elements, one block in the quotients of order $p$ and another in the quotients of order  $p^2$, as indicated in row 5 by $(\forall i)$. This is because we are taking $\psi_{h^p}$, the $p$-th power of the elements $h$ in $\psi_h$ of the first four rows, and so they are contained in $p$ groups of order $p^2$ in the genetic basis, as can be seen in the diagram of the previous section. It is easy to see that for any other element $h\in P$, $\psi_h$ evaluated in $a$ or $b$ will be a power of one of the rows of the table, so $\langle \textnormal{Im}\psi_h|h\in P\rangle$ is generated by these rows.

To see what the corresponding quotient is, first notice that, in the quotient, the element $a\langle a^{jp+i}b\rangle$ will be identified with $a^{-1}\langle a^ib,\, a^p\rangle$ and $b\langle ab^{jp}\rangle$ will be identified with $b^{-1}\langle a,\, b^p\rangle$. That is, the generators of the groups of order $p^2$ in $T$ will be identified with generators of the groups of order $p$ in $T$, so the corresponding quotient will be a quotient of the $p$-part of $T$: $P/\langle a,\, b^p\rangle\times \prod_{i=0}^{p-1}P/\langle a^ib, a^p\rangle$, which is isomorphic to $C_p^{p+1}$. 
Now, for a fixed $i$ take the $p$ rows corresponding to $0\leq j\leq p-1$ in the first one $(1,\, a\langle a^ib,\, a^p\rangle,\, 1,\, a\langle a^{jp+i}b\rangle)$ and multiply them all, this gives $(1, 1, 1, a\langle a^{jp+i}b\rangle\, (\forall j))$, where here at the end we have a block of size $p$ of elements of order $p^2$. Using row 5 we obtain then that the element $(1,\, a\langle a^ib,\, a^p\rangle\, (\forall i),\, 1,\, 1)$ is in $\langle \textnormal{Im}\psi_h |h \in P\rangle$. A similar argument with rows 4 and 8 will show that $(b\langle a,\, b^p\rangle,\, b\langle a^ib,\, a^p\rangle\, (\forall i),\, 1,\, 1)$ is also in $\langle \textnormal{Im}\psi_h |h \in P\rangle$. The use of the rest of the rows will add no information to what we have here, so the corresponding quotient is isomorphic to $P/\langle a,\, b^p\rangle\times \prod_{i=0}^{p-1}P/\langle a^ib, a^p\rangle$, over what is generated by the two elements $(1,\, a\langle a^ib,\, a^p\rangle\, (\forall i))$ and $(b\langle a,\, b^p\rangle,\, b\langle a^ib,\, a^p\rangle\, (\forall i))$, that is $SK_1(\ent P)\cong C_p^{p-1}$.

\subsubsection*{$C_{p^3}\times C_{p^3}$}

Let $P=C_{p^3}\times C_{p^3}$, and suppose it has generators $a$ and $b$. The genetic basis of $P$ is 
\begin{displaymath}
\{P, \,\langle a,\, b^p\rangle,\, \langle a^ib,\, a^p\rangle,\, \langle ab^{jp},\, b^{p^2}\rangle,\, \langle a^{jp+i}b,\, a^{p^2}\rangle,\, \langle ab^{kp^2+jp}\rangle,\, \langle a^{kp^2+jp+i}b\rangle\}
\end{displaymath}
with $0\leq i,\,j,\, k\leq p-1$. The group $T$ in this case is then isomorphic to
\begin{displaymath}
C_p^{\,p+1}\times C_{p^2}^{\,p(p+1)}\times C_{p^3}^{\,p^2(p+1)}.
\end{displaymath}
The table we obtain in this case is
\begin{displaymath}
\begin{array}{lr|cccccc}
 && \frac{P}{\langle a,\, b^p\rangle} & \frac{P}{\langle a^ib,\, a^p\rangle} & \frac{P}{\langle ab^{jp},\, b^{p^2}\rangle} & \frac{P}{\langle a^{jp+i}b,\, a^{p^2}\rangle} &\frac{P}{\langle ab^{kp^2+jp}\rangle} &\frac{P}{\langle a^{kp^2+jp+i}b\rangle} \\
 \cline{2-8}
p^3&\psi_{a^{kp^2+jp+i}b}(a) &1 &\bar{a} &1 & \bar{a}&1&\bar{a}\\
&\psi_{a^{kp^2+jp+i}b}(b) &1 &\bar{b} &1 & \bar{b}&1&\bar{b}\\
&\psi_{ab^{kp^2+jp}}(a) &1 &1 &\bar{a} & 1&\bar{a}&1\\
&\psi_{ab^{kp^2+jp}}(b) &\bar{b} & 1&\bar{b} & 1&\bar{b}&1\\
p^2&\psi_{a^{jp^2+ip}b^p}(a) & 1 &\bar{a} (\forall i) &1 & \bar{a} (\forall j)&1&\bar{a}(\forall k)\\
&\psi_{a^{jp^2+ip}b^p}(b) &\bar{b} &\bar{b} (\forall i) &1 & \bar{b} (\forall j)&1&\bar{b}(\forall k)\\
&\psi_{a^pb^{jp^2}}(a) &1 &\bar{a}(\forall i) &\bar{a}(\forall j) &1&\bar{a}(\forall k)&1\\
&\psi_{a^pb^{jp^2}}(b) &\bar{b} &\bar{b}(\forall i) &\bar{b}(\forall j) &1&\bar{b}(\forall k)&1\\
p&\psi_{a^{ip^2}b^{p^2}}(a) & 1 &\bar{a}(\forall i) &\bar{a}(\forall j) & \bar{a}(\forall j\forall i)&1&\bar{a}(\forall k\forall j)\\
&\psi_{a^{ip^2}b^{p^2}}(b) & \bar{b} &\bar{b}(\forall i) &\bar{b}(\forall j) & \bar{b}(\forall j\forall i)&1&\bar{b}(\forall k\forall j)\\
&\psi_{a^{p^2}}(a)&1&\bar{a}(\forall i) &\bar{a}(\forall j) & \bar{a}(\forall j\forall i) &\bar{a}(\forall k\forall j)&1\\
&\psi_{a^{p^2}}(b)&\bar{b}&\bar{b}(\forall i) &\bar{b}(\forall j) & \bar{b}(\forall j\forall i) &\bar{b}(\forall k\forall j)&1\\
1&\psi_1(a)&1&\bar{a}(\forall i) &\bar{a}(\forall j) & \bar{a}(\forall j\forall i) &\bar{a}(\forall k\forall j)&\bar{a}(\forall k\forall j\forall i)\\
&\psi_1(b)&\bar{b}&\bar{b}(\forall i) &\bar{b}(\forall j) & \bar{b}(\forall j\forall i) &\bar{b}(\forall k\forall j)&\bar{b}(\forall k\forall j\forall i)\\
\end{array}
\end{displaymath}
The first paragraph after the previous table can be extended to this one, adding that after row 9, we may have blocks of size $p^2$ of classes of elements, and in the last two rows there are blocks of size $p^3$.

As before, to see what the corresponding quotient is, first notice that, in the quotient, the element $a\langle a^{kp^2+jp+i}b\rangle$ will be identified with $a^{-1}\langle a^{jp+i}b,\, a^{p^2}\rangle a^{-1}\langle a^ib,\, a^p\rangle$ and $b\langle ab^{kp^2+jp}\rangle$ will be identified with $b^{-1} \langle ab^{jp},\, b^{p^2}\rangle b^{-1}\langle a,\, b^p\rangle$. That is, the generators of the groups of order $p^3$ in $T$ will be identified with products of the generators of  groups of order $p$ and $p^2$ in $T$, so the corresponding quotient will be a quotient of the part of $T$ isomorphic to $C_p^{\,p+1}\times C_{p^2}^{\,p(p+1)}$. Now, for fixed $i$ and $j$ take the $p$ rows corresponding to $0\leq k\leq p-1$ in the first one $(1,\, a\langle a^ib,\, a^p\rangle,\, 1,\,a\langle a^{jp+i}b,\, a^{p^2}\rangle,\, 1,\, a\langle a^{kp^2+jp+i}b\rangle)$ and multiply them all, this gives $(1,\, 1,\, 1,\, a^{p}\langle a^{jp+i}b,\, a^{p^2}\rangle,\, 1,\, a\langle a^{kp^2+jp+i}\rangle(\forall k))$. Next, write row 5 as 
\begin{displaymath}
(1,\, a\langle a^ib,\, a^p\rangle(\forall i),\, 1,\, a\langle a^{j_0p+i}b,\,a^{p^2}\rangle\cdots a\langle a^{j_{p-1}p+i}b,\,a^{p^2}\rangle,\, 1,\, a\langle a^{kp^2+jp+i}\rangle (\forall k) ).
\end{displaymath}
Now, for each $j_s$, multiply it by  the inverse of the element of the previous line with $j_s$, that is by
$(1,\, 1,\, 1,\, a^{-p}\langle a^{j_sp+i}b,\, a^{p^2}\rangle,\, 1,\, a^{-1}\langle a^{kp^2+jp+i}\rangle(\forall k))$, this gives that the element
\begin{equation}
(1,\, a\langle a^ib,\, a^p\rangle(\forall i),\, 1,\, a\langle a^{j_0p+i}b,\,a^{p^2}\rangle\cdots a^{1-p}\langle a^{j_sp+i}b,\, a^{p^2}\rangle\cdots a\langle a^{j_{p-1}p+i}b,\,a^{p^2}\rangle,\, 1,\, 1)
\end{equation}
is in $\langle\textnormal{Im}\psi_h\mid h\in P \rangle$. Taking the $p$-power of any of these elements, we obtain $(1,\, 1,\, 1,\, a^p\langle a^{j_0p+i}b,\,a^{p^2}\rangle\cdots a^p\langle a^{j_{p-1}p+i}b,\,a^{p^2}\rangle,\, 1,\, 1)$. By doing with row 5 as we did we row 1 we obtain  $(1,\, 1,\, 1,\, a^p\langle a^{j_0p+i}b,\,a^{p^2}\rangle\cdots a^p\langle a^{j_{p-1}p+i}b,\,a^{p^2}\rangle,\, 1,\, a\langle a^{kp^2+jp+i}b\rangle(\forall k\forall j))$. The use of these two elements allows us to obtain $(1,\, 1,\, 1,\, 1,\, 1,\, a\langle a^{kp^2+jp+i}b\rangle(\forall k\forall j))$, which by row 9 gives us that
\begin{equation}
(1,\, a\langle a^ib,\, a^p\rangle(\forall i),\, a\langle ab^{jp},\, b^{p^2}\rangle(\forall j),\, a\langle a^{jp+i}b,\, a^{p^2}\rangle(\forall j\forall i),\, 1,\,1)
\end{equation}
is in $\langle \textnormal{Im}\psi_h\mid h\in P\rangle$.

We do a similar process with the rows corresponding to the classes of $b$ and we obtain that for every $j_s$ the element
\begin{equation}
(b\langle a,\, b^p\rangle,\, b\langle a^ib,\, a^p\rangle (\forall i),\, b\langle ab^{j_0p},\,b^{p^2}\rangle\cdots b^{1-p}\langle ab^{j_sp},\, b^{p^2}\rangle\cdots b\langle ab^{j_{p-1}p},\,b^{p^2}\rangle,\, 1,\, 1,\, 1)
\end{equation}
is in $\langle \textnormal{Im}\psi_h\mid h\in P\rangle$, as well as the element
\begin{equation}
(1,\, b\langle a^ib,\, a^p\rangle(\forall i),\, b\langle ab^{jp},\, b^{p^2}\rangle(\forall j),\, b\langle a^{jp+i}b,\, a^{p^2}\rangle(\forall j\forall i),\, 1,\,1).
\end{equation}
It is not hard to see that the rest of the rows will add no information to what we have here. Doing some more calculations with (1), (2), (3) and (4), will show that the corresponding quotient is isomorphic to $C_p^{\,p+1}\times C_p^{\,(p+1)(p-2)}\times C_{p^2}^{\,p-1}$. That is $SK_1(\ent P)$ is isomorphic to $C_p^{\,p^2-1}\times C_{p^2}^{\,p-1}$.

\subsubsection*{$C_{p^4}\times C_{p^4}$}

As can be seen from the previous two cases, the process to obtain the quotient can be generalised as $n$ grows. Nonetheless, the relations we obtain from the table start getting more and more complicated. Doing a similar table as before and observing what are the relations we will have in $\langle \textnormal{Im}\psi_h\mid h\in P\rangle$, we obtain the following result for $SK_1$ of $C_{p^4}\times C_{p^4}$:
\begin{displaymath}
C_p^{\,(p^2+2)(p-1)}\times C_{p^2}^{\,2p(p-1)}\times C_{p^3}^{\,p-1}.
\end{displaymath}

\subsubsection*{A conjecture for $C_{p^n}\times C_{p^n}$}

For $n=5$ and $6$, the relations of the table become very complicated, but with the GAP \cite{gap} program that can be found in Appendix I, we obtained the following results for $p=3$:

$SK_1$ of $C_{p^5}\times C_{p^5}$ is isomorphic to
\begin{displaymath}
C_p^{\,(p^3+3)(p-1)}\times C_{p^2}^{\,(2p^2+p)(p-1)}\times C_{p^3}^{\,2p(p-1)}\times C_{p^4}^{\, p-1}.
\end{displaymath}

$SK_1$ of $C_{p^6}\times C_{p^6}$ is isomorphic to
\begin{displaymath}
C_p^{\,(p^4+4)(p-1)}\times C_{p^2}^{\,2(p^3+p)(p-1)}\times C_{p^3}^{\,p^3(p-1)}\times C_{p^4}^{\, 2p(p-1)}\times C_{p^5}^{\, p-1}.
\end{displaymath}

With these results we state the following conjecture.

\begin{conj}
Let $p$ be an odd prime and $i$ be a positive integer. For $n\geq 2i$ define: 
\begin{displaymath}
T_i(n)=(p-1)(2^{E(i)}p^{n-(\left \lfloor{i/2}\right \rfloor +2)}+(n-2i)p^{i-1}), 
\end{displaymath}
%and if $i$ is even
%\begin{displaymath}
%T_i(n)=(p-1)(2p^{n-(i+4)/2}+(n-2i)p^{i-1}). 
%\end{displaymath}
where $E(i)$ is $1$ if $i$ is even and $0$ if it is odd, and $\left \lfloor{i/2}\right \rfloor$ is the floor of $i/2$. %Also define $T_0(n)=0$ for every $n\geq 0$.

Then, for an integer $n\geq 2$ and $p$ an odd prime, the multiplicity of $C_{p^i}$, for $0<i<n$, in the decomposition of $SK_1(\ent [C_{p^n}\times C_{p^n}])$ as a product of cyclic groups, is $T_i(n)$ if $2i\leq n$ and $T_{n-i}(2(n-i))$ otherwise.
\end{conj}

By Example 5 in \cite{bob}, we know that the exponent of $SK_1(\ent [C_{p^n}\times C_{p^n}])$ is $p^{n-1}$, so the multiplicity of $C_{p^n}$ is $0$.

In particular, the conjecture says that the multiplicity of $C_p$ in \mbox{$SK_1(\ent [C_{p^n}\times C_{p^n}])$}, for $n\geq 2$, is $T_1(n)=(p^{n-2}+n-2)(p-1)$, as can be seen in the examples. It also says that the multiplicity of $C_{p^{n-1}}$ is always $T_{1}(2)=p-1$. 

%Many other examples show that this same method should work for any abelian $p$-group with $p$ odd FISH...

\subsection{$SK_1$ of $M_n(p)$}
If $n=3$, by Example 9.9  in \cite{bob} $SK_1(\ent M_n(p))\cong C_p^{\,p-1}$. The same example shows that $SK_1(\ent M_n(p))$ is non-trivial for all $n\geq 4$. 

\begin{teo}
For $n\geq 4$, the group $SK_1(\ent M_n(p))$ is isomorphic to $C_p^{(n-2)(p-1)}$.
\end{teo}
\begin{proof}
We will write $P$ for $M_n(p)$ and $\mathcal{S}$ for the genetic basis we found in the previous section. Except for $B$, all the groups in $\mathcal{S}$ are normal in $P$, so their corresponding simple modules are just inflations to $P$, thus they can be treated as in the abelian cases.

For $B$, let $V$ be the corresponding simple module $V=\ind_{A_1B}^P\ifl_{A_1B/B}^{A_1B}W$ where $W=\Phi_{A_1B/B}$. Let $h$ be an element in $P$, we will determine $V^h$. Let $H=\langle h\rangle$. By Mackey's formula we have
\begin{displaymath}
\res_H^P V \cong  \bigoplus_{x\in [\frac{P}{HA_1B}]}\ind_{H\cap A_1B}^H\res_{H\cap A_1B}^{A_1B}(\ifl_{A_1B/B}^{A_1B} W)^x,
\end{displaymath}
so $(\res_H^PV)^H$ is isomorphic to
\begin{align*}
(\res_H^PV)^H&\cong\bigoplus_{x\in [\frac{P}{HA_1B}]}Hom_{\rac H}(\rac,\, \ind_{H\cap A_1B}^H\res_{H\cap A_1B}^{A_1B}(\ifl_{A_1B/B}^{A_1B} W)^x)\\
&\cong\bigoplus_{x\in [\frac{P}{HA_1B}]}Hom_{\rac (H\cap A_1B)}(\rac,\, \res_{H\cap A_1B}^{A_1B}(\ifl_{A_1B/B}^{A_1B} W)^x)\\
&\cong \bigoplus_{x\in [\frac{P}{HA_1B}]}(\res_{H\cap A_1B}^{A_1B}(\ifl_{A_1B/B}^{A_1B} W)^x)^{H\cap A_1B}.
\end{align*}
If $h$ is not in $A_1B$, then $HA_1B=P$ and we have only one summand, so we are looking for the fixed points of  $\ifl_{A_1B/B}^{A_1B} W$ by ${H\cap A_1B}$. We claim that in this case $H\cap A_1B=A_1$ and so the set of fixed points is $\{0\}$. Since the index of $H\cap A_1B$  in $H$ is $p$ and the order of $A_1$ is $p^{n-2}$, it suffices to see that $A_1$ is contained in $H$. %Observe that the order of $H\cap A_1B$ is $p^{n-2}$. 
Suppose that $h=a^kb^l$ with $(p,\, k)=1$. If $p$ divides $l$ then the claim is clear, so suppose also that $(p,\, l)=1$. Since $P$ is regular (Lemma 2.9 in \cite{sim}), we have that $h^p=a^{kp}c$, where $c$ is in $\mho_1(\langle a^k,\, b^l\rangle')=\{y^p\mid y \in \langle a^k,\, b^l\rangle'\}$, but $\langle a^k,\, b^l\rangle=P$ so we have $\mho_1(\langle a^k,\, b^l\rangle')=1$. This finishes this case.

If $h$ is in $A_1B$, then $H\cap A_1B=H$. To find the fixed points of $(\ifl_{A_1B/B}^{A_1B} W)^x$ by $H$ we conjugate and find the fixed points of $M=\ifl_{A_1B/B}^{A_1B} W$ by ${^x\!H}$. This module is different from $0$ if and only if $^x\!H$ is contained in $B$. If $h\neq 1$ this is equivalent to say that $^xH=B$ and $M\cong W$. 

From this we conclude that $V^H$ is different from zero only when $H=1$ or when $H=B^x$ for some $x\in \{1,a,\ldots,a^{p-1}\}$ and in this case $V^H$ is isomorphic to $W^x$.

Now we proceed to calculate the centralizers of elements in $P$.

It is easy to see that $Z(P)$, the center of $P$, is generated by $a^p$, so the centralizer of any element in $A_1$ is $P$. Now, for any other element $h$ in $A_1B$ and not in $A_1$, its centralizer is clearly $A_1B$. Finally, if $h$ is not in $A_1B$ then, as seen before, it generates a group $H$ of index $p$ and so $H$ is the centralizer of $h$.

With this information we can make a table as we did for the abelian cases.

We will make the table for the case $n=5$, this case illustrates very well how the table is in general.
\begin{displaymath}
\begin{array}{lr|cccccccc}
 && \frac{P}{\langle a\rangle} & \frac{P}{\langle a^ib\rangle} & \frac{P}{\langle a^{p},\, b\rangle} & \frac{P}{\langle a^{ip}b\rangle} &\frac{P}{\langle a^{p^2},\, b\rangle} &\frac{P}{\langle a^{ip^2}b\rangle} & \frac{P}{\langle a^{p^3},\, b \rangle} &\frac{A_1B}{B} \\
 \cline{2-10}
p^4 & \psi_{a^ib}(a^ib)&1&1&1&1&1&1&1&1\\
& \psi_a(a) &1&1&1&1&1&1&1&1\\
p^3&\psi_{a^{ip}b}(a^p) &1 &1 &1 & \bar{a}^p_i&1&1&1&1\\
&\psi_{a^{ip}b}(b) &1 &1 &1 & \bar{b}_i&1&1&1&1\\
&\psi_{a^p}(a) &1 &\bar{a} (\forall i) &\bar{a} & 1&1&1&1&1\\
&\psi_{a^p}(b) &\bar{b} &\bar{b}(\forall i) &1 & 1&1&1&1&1\\
p^2&\psi_{a^{ip^2}b}(a^p) &1 &1 &1 & 1&\bar{a}^p&\bar{a}^p_i&1&1\\
&\psi_{a^{ip^2}b}(b) &1 &1 &1 & 1&1&\bar{b}_i&1&1\\
&\psi_{a^{p^2}}(a) &1 &\bar{a} (\forall i) &\bar{a} & \bar{a} (\forall i)&\bar{a}&1&1&1\\
&\psi_{a^{p^2}}(b) &\bar{b} &\bar{b} (\forall i) &1 & \bar{b} (\forall i)&1&1&1&1\\
p&\psi_{b}(a^p) &1 &1 &1 & 1&\bar{a}^p&1&\bar{a}^p&\bar{a}^p\\
&\psi_{b}(b) &1 &1 &1 & 1&1&1&1&1\\
&\psi_{b^x}(a^p) &1 &1 &1 & 1&\bar{a}^p&1&\bar{a}^p&\bar{a}^p\\
&\psi_{b^x}(b) &1 &1 &1 & 1&1&1&1&1\\
&\psi_{a^{p^3}}(a) & 1 &\bar{a} (\forall i) &\bar{a} &\bar{a}(\forall i)&\bar{a}& \bar{a} (\forall i)&\bar{a}&1\\
&\psi_{a^{p^3}}(b) & \bar{b} &\bar{b} (\forall i) &1 &\bar{b}(\forall i)&1& \bar{b} (\forall i)&1&1\\
1&\psi_1(a) &1&\bar{a} (\forall i)&\bar{a} &\bar{a} (\forall i)&\bar{a} &\bar{a} (\forall i)&\bar{a} &\bar{a}^p \\
&\psi_1(b) &\bar{b}&\bar{b} (\forall i)&1 &\bar{b} (\forall i)&1 &\bar{b} (\forall i)&1 &\bar{a}^{-p^3}
\end{array}
\end{displaymath}
In rows 3 and 4, what we mean by $\bar{a}^p_i$ (and $\bar{b}_i$) is that $\bar{a}^p$ is only in the entry corresponding to the $i$ from $a^{ip}b$, to differentiate it from the notation $\bar{a} (\forall i)$.

To find the value of $\psi_1(a)$ in $A_1B/B$, let $\zeta_{p^{n-2}}$ be a primitive $p^{n-2}$-th root of unity and $a^j\otimes \zeta_{p^{n-2}}$, with $j$ running from 0 to $p-1$, be a basis over $\rac (\zeta_{p^{n-2}})$ of $V(B)=\ind_{A_1B}^P\ifl_{A_1B/B}^{A_1B}\Phi_{A_1B/B}$, then the matrix of the action of $a$ in this basis (with coefficients in $A_1B/B$) is 
\begin{displaymath}
\left( \begin{array}{cc}
0 & \bar{a}^p\\
Id_{p-1}& 0
\end{array}\right).
\end{displaymath}
where $Id_{p-1}$ is the identity matrix of size $p-1$. So, this matrix has determinant $\bar{a}^p$. The action of $b$ in this basis is given by the matrix
\begin{displaymath}
\left( \begin{array}{cc}
1 & 0\\
0 &\bar{a}^{p^{n-2}}Id_{p-1}
\end{array}\right).
\end{displaymath}
so $\psi_1(b)$ in $A_1B/B$ is $\bar{a}^{(p-1)p^{n-2}}$.

To finish the proof, we see what relations on $T=C_p^{p+1}\times C_{p^2}^p\times \cdots\times C_{p^{n-3}}^p\times C_{p^{n-2}}^{p+1}$ we obtain from the rows of the table. What we will do works in general or can be easily generalized for any $n\geq 4$. 

Consider then $T=C_p^{p+1}\times C_{p^2}^p\times C_{p^3}^{p+1}$. From rows 15 and 17, we see that $A_1B/B$, corresponding to the last $C_{p^3}$ will disappear in the quotient, so consider $T_1=C_p^{p+1}\times C_{p^2}^p\times C_{p^3}^p$. Rows 5 and 6 indicate that from $C_p^{p+1}$ we will obtain $C_p^{p-1}$ in the quotient. Rows 3 and 4 say that the the first $p-1$ factors of $C_{p^2}^p$ will transform into $C_p$, giving $C_p^{p-1}$. As for the last factor, $C_{p^2}$, from rows 5 and 9 we see that its generator $\bar{a}$ will be identified in the quotient with a product of generators coming from the other factors $C_{p^2}$. That is, from $C_{p^2}^p$ we obtain $C_p^{p-1}$ in the quotient. A similar procedure, using rows 7, 9 and 15 shows that from the part $C_{p^3}^p$ we obtain $C_p^{p-1}$. The rest of the rows add nothing more and we obtain $C_p^{3(p-1)}$. It is not hard to see that this procedure can be extended to any $n\geq 5$ and for the case $n=4$ is even simpler. So in general we will obtain $C_p^{(n-2)(p-1)}$.
\end{proof}

\section*{Acknowledgements}
Thanks a lot to Serge Bouc, for helping me with the GAP program.

\section*{Appendix: GAP program}

\begin{small}
\begin{verbatim}
# Computing SK1 of an abelian p-group (p odd)
#
# The (odd) prime number p
#
pr:=5;
#
# The p-group G, given by its decomposition as a product 
# of cyclic groups
#
exp:=[pr^3,pr^3];
Sort(exp);
exp:=Reversed(exp);
g:=AbelianGroup(exp);
Print("Group : ",StructureDescription(g),"\n");
#
# The exponent of G
#
eg:=exp[1];
#
# Generators of G
#
gg:=GeneratorsOfGroup(g);
lo:=exp;
nf:=Length(exp);
im:=AbelianGroup([eg]);
gim:=GeneratorsOfGroup(im);
gim:=gim[1];
#
# Computing genetic subgroups of G
#
ff:=List([1..nf],x->[0..lo[x]-1]);
ff[1]:=List([0..LogInt(eg,pr)],x->pr^x mod eg);
cff:=Cartesian(ff);
lsthom:=List(cff,s->GroupHomomorphismByImages
(g,im,gg,List([1..nf],x->gim^(eg/lo[x]*s[x]))));
lstker:=List(lsthom,x->Kernel(x));
sgp:=Set(lstker);
psgp:=List(sgp,x->Position(lstker,x));
cff:=cff{psgp};
Print("Genetic subgroups: ",Collected(List(sgp,Size)),"\n");
##
sgp:=Filtered(sgp,x->x<>g);
nsgp:=Length(sgp);
fquot:=List(sgp,x->NaturalHomomorphismByNormalSubgroup(g,x));
genquot:=List([1..nsgp],x->MinimalGeneratingSet(Image(fquot[x]))[1]);
#
# Initalizing an integral matrix with cokernel isomorphic to
# the direct product of non trivial cyclic quotients of G
#
mat:=[];
for i in [1..nsgp] do
    v:=List([1..nsgp],x->0);
    v[i]:=Index(g,sgp[i]);
    Add(mat,v);
od;
#
# Adding rows to the matrix, corresponding to relations in SK1
#
lcff:=Length(cff);
for iih in [1..lcff] do
    ih:=cff[iih];
    Print("\rComputing SK1... ",iih,"/",lcff,"\r");
    h:=Product([1..nf],x->gg[x]^ih[x]);
    pos:=Filtered([1..nsgp],x->h in sgp[x]);
    for a in gg do
        v:=List([1..nsgp],x->0);
        v{pos}:=List(pos,x->First([0..eg],
        y->genquot[x]^y=Image(fquot[x],a)));
        if not v in mat then
            Add(mat,v);
        fi;
    od;
od;
Print("Decomposition of SK1                           \r");
tt:=Collected(ElementaryDivisorsMat(Integers,mat));
tt:=Filtered(tt,x->x[1]<>1);
Print("\r                              \rSK1 = ");
if tt=[] then
    Print("0\n");
else
    for u in [1..Length(tt)-1] do
        Print("(C",tt[u][1],")^",tt[u][2]," x ");
    od;
    Print("(C",tt[Length(tt)][1],")^",tt[Length(tt)][2],"\n");
fi;
\end{verbatim}
\end{small}

\bibliographystyle{plain}
\bibliography{whitehead}

\begin{thebibliography}{10}

\bibitem{bas}
B.G. Basmaji.
\newblock Complex representations of metacyclic groups.
\newblock {\em The American Mathematical Monthly}, 86:47--48, 1979.

\bibitem{bouc}
Serge Bouc.
\newblock {\em Biset functors for finite groups}.
\newblock Springer, Berlin, 2010.

\bibitem{curtis1}
{C. W.} Curtis, {I.}~Reiner.
\newblock {\em Representation theory of finite groups and associative
  algebras}.
\newblock John Wiley and sons, U.S.A, 1962.

\bibitem{gap}
GAP.
\newblock http://www.gap-system.org/.

\bibitem{john}
John Guaschi, Daniel Juan-Pineda, and Silvia Mill\'an.
\newblock The lower algebraic ${K}$-theory of the braid groups of the sphere.
\newblock {\em Preprint}, 2012.

\bibitem{magu}
Jean-Fran\c{c}ois Lafont, Bruce~{A}. Magurn, and Ivonne~{J}. Ortiz.
\newblock Lower algebraic ${K}$-theory of certain reflection groups.
\newblock {\em Proceedings of the Cambridge Philosophical Society},
  148:193--226, 2010.

\bibitem{bob}
Robert Oliver.
\newblock {\em Whitehead groups of finite groups}.
\newblock Cambridge University Press, U.K., 1988.

\bibitem{arith}
Jean-Pierre Serre.
\newblock {\em Cours d'arithm\'etique}.
\newblock Presses universitaires de France, Paris, 1970.

\bibitem{serre}
Jean-Pierre Serre.
\newblock {\em Linear representations of finite groups}.
\newblock Springer-Verlag, New York, 1977.

\bibitem{sim}
H.~S. Sim.
\newblock Metacyclic groups of odd order.
\newblock {\em Proc. London Math. Soc.}, 69(3):47--71, 1994.

\bibitem{ushi}
F.~Ushitaki.
\newblock A generalization of a theorem of {M}ilnor.
\newblock {\em Osaka Journal of Mathematics}, 31:430--415, 1994.

\end{thebibliography}

\end{document}